\documentclass[a4paper,english,numbers=noenddot,fontsize=10pt]{article}

\usepackage{amsmath}             % advanced math extension
\usepackage{amssymb}             % new math symbols
\usepackage{amsthm}              % theoremstyle und proof-Umgebung
\usepackage{amsfonts}
\usepackage{array}
\usepackage{color}              % coloured text
\usepackage{graphicx}            % external pictures
\usepackage{lmodern}
\usepackage[automark]{scrpage2}
\usepackage[T1]{fontenc}
\usepackage{algorithm}
\usepackage[noend]{algorithmic}
\usepackage{cite}
\usepackage{rotating}
\usepackage{makeidx}
\usepackage{titlesec}
\usepackage{lscape}
\usepackage{threeparttable}
\usepackage{calc}
\usepackage{multicol}
\usepackage{enumitem}
\usepackage{chemarrow}
\usepackage{dcolumn}
\usepackage{url}

\if@Mn@Math@
 \DeclareMathSymbol{\ell}{\mathlord}{letters}{'140}
\fi

\DeclareSymbolFont {mysymbols}{OMS}{cmsy}{m}{n} 
\DeclareMathSymbol{\calI}{\mathalpha}{mysymbols}{`I}
\DeclareSymbolFont {mylargesymbols}{OMX}{cmex}{m}{n}

\DeclareMathOperator{\lm}{lm}
\DeclareMathOperator{\lt}{lt}
\DeclareMathOperator{\lc}{lc}

\DeclareMathOperator{\lcm}{lcm}

\DeclareMathOperator{\idx}{index}

\DeclareMathOperator{\syz}{Syz}

\DeclareMathOperator{\poly}{poly}

\DeclareMathOperator{\ggv}{G2V}
\DeclareMathOperator{\ggvc}{iG2V}

\DeclareMathOperator{\ff}{F5}

\DeclareMathOperator{\ffc}{F5C}
\DeclareMathOperator{\ffcc}{iF5C}

\DeclareMathOperator{\ffe}{F5A}
\DeclareMathOperator{\ffec}{iF5A}

\DeclareMathOperator{\siggb}{\textsc{SigGB}}

\DeclareMathOperator{\incsig}{\textsc{IncSig}}

\DeclareMathOperator{\incsigc}{\textsc{IncSigR}}

\DeclareMathOperator{\redsb}{\textsc{RedSB}}
\DeclareMathOperator{\reduce}{\textsc{Reduce}}

\newcommand*{\openRight}{\csname .openrighttrue\endcsname}

\newcommand\rp{\ensuremath{\mathcal{R}}}

\newcommand\psyz{\ensuremath{\textrm{PSyz}}}

\newcommand\sigvar{e}
\newcommand\siglt{\ensuremath{\prec}}
\newcommand\sigle{\ensuremath{\preceq}}

\newcommand\monset{\ensuremath{\mathcal{M}}}
\newcommand\sigset{\ensuremath{\mathcal{S}}}

\newcommand\spol[2]{\ensuremath{\mathrm{spoly}(#1,#2)}}
\newcommand\sig{\ensuremath{\mathrm{sig}}}

\newcommand\field{\ensuremath{\mathcal{K}}}

\newcommand\pring{\ensuremath{\mathcal{K}[x_1,\ldots,x_n]}}

\definecolor{commentcolor}{rgb}{0.6, 0.6, 0.6}
\definecolor{newcolor}{RGB}{218,22,20}
\definecolor{1}{RGB}{16,78,139}
\definecolor{2}{RGB}{0,114,49}
\definecolor{3}{RGB}{96,139,0}
\definecolor{4}{RGB}{255,100,0}
\definecolor{5}{RGB}{255,49,0}
\definecolor{6}{RGB}{125,0,0}

\floatstyle{ruled}
\newfloat{algorithm}{tbp}{loa}
\floatname{algorithm}{Algorithm}

\numberwithin{equation}{section} %% Comment out for sequentially-numbered
\numberwithin{figure}{section} %% Comment out for sequentially-numbered

\theoremstyle{plain}
\newtheorem{thm}{Theorem}[section]
\newtheorem{prop}[thm]{Proposition}
\newtheorem{cor}[thm]{Corollary}
\newtheorem{lem}[thm]{Lemma}

\theoremstyle{definition}
\newtheorem{defi}[thm]{Definition}
\newtheorem{example}[thm]{Example}
\theoremstyle{remark}
\newtheorem{rem}[thm]{Remark}

\definecolor{commentcolor}{rgb}{0.6, 0.6, 0.6}

\makeatother

\usepackage{babel}
\begin{document}
%\input{nocites}
%\conferenceinfo{ISSAC2012,} {July 22--25, 2012, Grenoble, France.}
%\CopyrightYear{2012} \crdata{XXX} \clubpenalty=10000 \widowpenalty = 10000

\title{Improving incremental signature-based Gr\"obner basis
  algorithms}

\author{%
Christian Eder\\
c/o Department of Mathematics \\
TU Kaiserslautern \\
67653 Kaiserslautern, Germany \\
ederc@mathematik.uni-kl.de
}
\maketitle

\begin{abstract}
In this paper we describe a combination of ideas to improve incremental signature-based
Gr\"obner basis algorithms having a big impact on their performance.
Besides explaining how to combine already known optimizations to achieve more
efficient algorithms, we show how to improve them even more. Although our idea 
has a psotive affect on all kinds of incremental signature-based algorithms, the 
way this impact is achieved can be quite different. Based on the two best-known 
algorithms in this area, $\ff$ and $\ggv$, we explain our idea, both from a 
theoretical and a
practical point of view.
\end{abstract}

%\category{I.1.2}{Symbolic and Algebraic Manipulation}{Algorithms}[Algebraic Algorithms]
%\category{F.2.1}{Analysis of Algorithms and Problem Complexity}{Numerical Algorithms and Problems}[Computations on Polynomials]
%\terms{Algorithms}

%testing git right now!
%\keywords{Gr\"obner bases, F5 Algorithm, $\ggv$\ Algorithm, incremental signature-based
%  algorithms}

\section{Introduction}

Computing Gr\"obner bases is a fundamental tool in commutative and computer
algebra. Buchberger introduced the first algorithm to compute such bases in
1965, see~\cite{bGroebner1965}. In the meantime lots of additional and
improved algorithms have been developed.

In the last couple of years, so-called
{\it signature-based} algorithms like $\ff$, see~\cite{fF52002Corrected}, and
$\ggv$, see~\cite{ggvGGV2010}, have become more popular. 
Lots of optimizations for these algorithms have been published, for
example, see~\cite{epF5C2009,apF5CritRevised2011,gvwGVW2010,sw2011}.
Whereas the last two publications focus on the field of non-incremental signature-based algorithms,
we focus our discussion in this paper to the initial presentation of this
kind of algorithms, based on Faug\`ere's initial presentation of $\ff$
in~\cite{fF52002Corrected}: Computing Gr\"obner bases step by step iterating
over the generators of the input system. The intermediate states of this 
incremental structure can be used to improve performance.

The intention of this paper is not only to cover, to collect, and to compare the various
optimizations found recently, but also to increase the algorithms'
efficiency.
As discussed in-depth in~\cite{epSig2011}, signature-based algorithms
differ mainly on their implementation of two criteria used to detect useless
critical pairs during the computations, the {\it non-minimal
signature} criterion and the {\it rewriting signature} criterion; the
optimizations presented in this publication have mostly an impact on the first 
one criterion. We focus our discussion on the two best-known and most
efficient incremental algorithms in this area, namely $\ff$ and $\ggv$. Due to 
their different, in some sense even opposed, usages of the above mentioned 
criteria, their behaviour under our optimizations gives a rather accurate 
picture of the general impact on the class of incremental signature-based 
algorithms.

In Section~\ref{sec:Background} we introduce the basic notions of
incremental signature-based algorithms.
In~\cite{epF5C2009} the idea of interreducing intermediate Gr\"obner bases
between the iteration steps of $\ff$ is illustrated: Speed-ups of up 
to $30\%$ comparing $\ffc$ to the basic $\ff$ can be achieved by minimizing the computational overhead
generated due to the inner workings of signature-based algorithms.
Section~\ref{sec:ffc} shortly reviews this idea, from
a more general point of view than it was done in its initial
presentation back then, taking its effects on algorithms like $\ggv$ into account.
$\ggv$ and the idea of using zero reductions actively in the current iteration
step is content of Section~\ref{sec:ggv}. The idea of using recent zero 
reductions in the algorithm goes back to Alberto Arri's preprint of 
\cite{apF5CritRevised2011} in 2009, where such a kind of optimization was 
mentioned for the first time.  Combining these two, at a first look quite 
separated improvements in a clever
way is the main contribution of this
paper: We see that a quite easy idea can be used to get a faster detection of
useless critical pairs; in the situation of $\ggv$ one even discards more
elements, which leads to a huge improvement in the overall performance of the
algorithm.

\section{Basic setting}
\label{sec:Background}
We start with some basic notations.
Let $i\in\mathbb{N}$, $\field$ a field, and
$\rp=\pring$. Let $F_{i}=\left(f_{1},\ldots,f_{i}\right)$,
where each $f_{j}\in \rp$, and $I_{i}=\langle F_{i}\rangle \subset \rp$ is 
the ideal generated by the elements of $F_{i}$. Moreover, we fix a {\it degree-compatible
ordering $<$} on the monoid $\monset$ of monomials of $x_{1},\ldots,x_{n}$.
For a polynomial $p\in \rp$, we denote $p$'s {\it leading monomial} by $\lm\left(p\right)$,
its {\it leading coefficient} by $\lc\left(p\right)$, and write
$\lt\left(p\right)=\lc\left(p\right)\lm\left(p\right)$ for its {\it leading
term}. For any two polynomials
$p,q \in \rp$ we use the shorthand notation
\[ \tau(p,q) = \lcm\left(\lm(p),\lm(q)\right)\]
for the {\it least common multiple} of their leading monomials.

Let $\sigvar_{1},\ldots,\sigvar_{m}$ be the canonical generators
of the free $\rp$-module $\rp^{m}$. We extend the ordering $<$ to a
well-ordering $\prec$ on the set $\left\{ t e_j \mid~t\in \monset, 1\leq j \leq
m\right\}$ in the following way\footnote{Note that this differs slightly from
  the ordering given in~\cite{fF52002Corrected}, but our discussion is mainly based
    on~\cite{epSig2011}. Moreover, it simplifies notation.}:
$t_j\sigvar_{j}\siglt t_k\sigvar_{k}$
iff $j<k$, or $j=k$ and $t_j < t_k$.
We define maps
\begin{center}
$
\begin{array}{rccc}
\pi : &\rp^{i} &\rightarrow &I_{i}\\
    &\sum_{j=1}^{i} p_je_j &\mapsto &\sum_{j=1}^i p_jf_j,
\end{array}
$
\end{center}
where $p_j$ is a polynomial in $\rp$ for $1\leq j \leq i$. An element $\omega \in
\rp^{i}$ with $\pi(\omega) = 0$ is called a {\it syzygy of $f_1,\dotsc,f_{i}$}. The
module of all such syzygies is denoted $\syz(F_i)$. A syzygy of type
$f_k e_j - f_j e_k$ is called a {\it principal syzygy}. The submodule of all principal
syzygies of $F_i$ is denoted $\psyz(F_i) \subseteq
\syz(F_i)$. Note that if a sequence $F_i$ of polynomials is regular, then
$\psyz(F_i) = \syz(F_i)$.

Let $f_{i+1}\in \rp \backslash I_{i}$. We describe algorithms that, given a Gr\"obner basis $G_{i}$ of $I_{i}$, computes a
Gr\"obner basis of $I_{i+1}=\langle F_{i+1}\rangle$, where $F_{i+1}=\left(f_{1},\ldots,f_{i+1}\right)$.
Thus we restrict ourselves to an incremental approach in this paper.

In \cite{epSig2011} the class of signature-based Gr\"obner basis algorithms is
introduced. Those give a new point of view on the computations taking so-called
{\it signatures} into account.

\begin{defi}
\label{def:sig}
 Let $p\in I_{i+1}$, $j\in\mathbb{N}$
with $j\leq i+1$, and $h_{1},\ldots,h_{j}\in \rp$ such that $h_{j}\neq0$
and\[
p=h_{1}f_{1}+\cdots h_{j}f_{j}.\]
\begin{enumerate}
\item If $t=\lm\left(h_{j}\right)$, we
say that $t\sigvar_{j}$ is {\em a signature} of $p$. Let $\sigset$
be the set of all signatures:\[
\sigset=\left\{ t\sigvar_{j} \mid ~t\in\monset,~j=1,\ldots,i+1\right\} .\]
\item Using the well-ordering $\siglt$ on $\sigset$ we can identify for each $p\in
\rp$ a unique, minimal signature.
\item For any monomial $u \in \monset$ and any signature $te_j \in \sigset$ we
define a multiplication $u \cdot te_j := (ut)e_j$.
\item An element $f = (te_j,p) \in \sigset \times I_{i+1}$ is called a {\em labeled
  polynomial}. For a labeled polynomial $f=(te_j,p)$ we define the shorthand
  notations $\poly(f) = p$, $\sig(f) = te_j$, and $\idx(f) = j$. Talking about the leading
  monomial, leading term, and leading coefficient of a labeled polynomial $f$ we
  always assume the corresponding value of $\poly(f)$. In the same sense we
  define the least common multiple of two labeled polynomials $f$ and $g$,
  $\tau(f,g)$, by $\tau\left(\poly(f),\poly(g)\right)$. Furthermore, if $G =
  \{g_1,\dotsc,g_{\ell}\} \subset \sigset \times I_{i+1}$, then we define
  \[\poly(G) := \left\{\poly(g_1),\dotsc,\poly(g_{\ell})\right\} \subset I_{i+1}.\]
\item A {\em critical pair of two labeled polynomials $f$ and $g$} is a tuple $(f,g) \in \left(\sigset \times I_{i+1}\right)^2$.
\item Moreover, we define the
{\em s-polynomial of two labeled polynomials $f$ and $g$} by 
\[\spol{f}{g}=\left(\omega,u_f\cdot \poly(f)-\frac{\lc(f)}{\lc(g)}u_g\cdot
    \poly(g)\right)\]
where
\begin{enumerate}
\item $u_f =\frac{\tau(f,g)}{\lm\left(f\right)}$, $u_g
=\frac{\tau(f,g)}{\lm\left(g\right)} \in \monset$, and 
\item $\omega = \max\left\{
u_f\sig(f),u_g\sig(g) \right\}$. 
\end{enumerate}
\end{enumerate}
\end{defi}

Adopting the notions of reduction and standard representation from the pure polynomial setting
we get:
\begin{defi}
Let $f,g \in \sigset \times I_{i+1}$ be labeled polynomials, and let $G\subset \sigset
\times I_{i+1}$, with $\#G=\ell$.
\begin{enumerate}
\item {\em $f$ reduces sig-safe to $g$ modulo $G$} if there exist sequences 
$j_{1},\ldots,j_{\ell}\in\mathbb{N}$, $t_{1},\ldots,t_{\ell}\in\monset$,
  $c_{1},\ldots,c_{\ell}\in\field$,
and $r_{0},\ldots,r_{\ell}\in \sigset \times I_{i+1}$ such that for all $i\in
\{1,\ldots,\ell\}$ $g_{j_{i}} \in G$,
\begin{enumerate}
\item $r_{0} = f$, $r_{i}=r_{i-1}-c_{i}t_{i}g_{j_{i}}$, $r_{l} =g$,
\item $\lm\left(r_{i}\right)<\lm\left(r_{i-1}\right)$, and
\item $t_i\sig\left(g_{j_{i}}\right)\siglt\sig\left(r_{i-1}\right)$.
\end{enumerate}
\item We say that $f$ has a {\em standard representation with respect to $G$} if there
exist $h_{1},\ldots,h_{\ell}\in \rp$, $g_1,\dotsc,g_{\ell} \in G$  such that
\begin{enumerate}
\item $\poly(f)=h_{1}\poly(g_{1})+\cdots+h_{\ell}\poly(g_{\ell})$, 
\item for each $k=1,\ldots,\ell$ either $h_{k}=0$, or
\begin{enumerate}
\item $\lm\left(h_{k}\right)\lm\left(g_{k}\right)\leq\lm\left(f\right)$, and 
\item $\lm\left(h_{k}\right)\sig(g_{k})\sigle \sig(f)$. 
\end{enumerate}
\end{enumerate}
\end{enumerate}
\end{defi}
\begin{rem}\
\begin{enumerate}
\item If $f$ reduces sig-safe to $g$ modulo $G$, then it has a standard representation
modulo $G$. Moreover, note that the concept of sig-safeness, that means the
restriction of the reducer $g_{j_{i}}$ by $t_i\sig\left(g_{j_{i}}\right)\siglt\sig\left(r_{i-1}\right)$ in
each step, is essential for the correctness (and the performance) of signature-based algorithms.  
\item In Fact~24 of \cite{epSig2011} it is shown that it is sufficient to
consider signatures with coefficient $1$. Thus there is no need to consider
$\lc(h_{j})$ in Definition~\ref{def:sig} resp. terms in general for signatures.
\end{enumerate}
\end{rem}

The following statement is the signature-based counterpart of Buchberger's
Criterion, see~\cite{bGroebner1965}.

\begin{thm}
\label{thm:gb} 
Let $G_{i+1} = \left\{g_1,\dotsc,g_{\ell}\right\}
\subset \sigset \times I_{i+1}$ such that $\left\{f_1,\dotsc,f_{i+1}\right\}
\subset \poly(G_{i+1})$. If for each pair $(j,k)$ with $j>k$, $1\leq j,k
\leq \ell$, $\spol{g_j}{g_k}$ has a standard representation w.r.t. $G_{i+1}$,
  then $\poly(G_{i+1})$ is a Gr\"obner basis of $I_{i+1}$.
\end{thm}

\begin{proof}
For example, see~\cite{epF5C2009,epSig2011}.
\end{proof}

In the non-signature-based setting, an algorithm based
on the Buchberger Criterion only would be quite inefficient.
There the Product Criterion and the Chain Criterion,
      see~\cite{bGroebner1965,kollBuchberger1978} are used to reduce
useless computations; a notable implementation can be found
in~\cite{gmInstallation1988}.
On the signature-based side the very same holds: We need criteria to improve the
computations, see~\cite{epSig2011} for more details on this topic.

In the following we assume an incremental signature-based Gr\"obner basis
algorithm, for example, the reader can think of $\ff$'s
presentation in~\cite{fF52002Corrected}, $\ggv$ as stated in~\cite{ggvGGV2010},
or the more general discussion in~\cite{epSig2011}. 

Next we state the basic criteria used in $\ff$. It is important to start from
$\ff$'s point of view to see how optimizations in the signature-based world can
be achieved. Note that we use the notations introduced in~\cite{epSig2011} for
an easier adoption to incremental signature-based algorithms in general
later on.

\begin{lem}
\label{lem:criteria}
Assume the computation of a Gr\"obner basis $\poly(G_{i})$ for $I_{i}$, that means we are in
the $i$th incremental step of $\ff$.
\begin{enumerate}
\item\label{nmcriterion} {\em Non-minimal signature criterion (NM)}: $\spol{f}{g}$ has a standard
representation w.r.t. $G_{i}$ if there exists an element $\omega \in
\psyz(F_{i})$ with $\lm(\omega) = t_{\omega}e_{j_{h}}$ such that 
\begin{enumerate}
\item $t_{\omega} \mid u_h t_h$,
\end{enumerate}
where $u_h \sig(h) = u_h t_h e_{j_{h}}$ for either $h=f$ or $h=g$. 
\item {\em Rewritable signature criterion (RW)}: $\spol{f}{g}$ has a standard
representation w.r.t. $G_{i}$ if there exists a labeled polynomial $r$ such that 
\begin{enumerate}
\item $\idx(r) = \idx(h)$,
\item $\sig(r) \succ \sig(h)$, and
\item $t_r \mid u_h t_h$,
\end{enumerate}
where $\sig(r)=t_r e_{j_h}$ and $u_h \sig(h) = u_h t_h e_{j_h}$ for either $h=f$ or $h=g$.
\end{enumerate}
\end{lem}

\begin{proof}
The two statements of Lemma~\ref{lem:criteria} are included in Theorem~18
of~\cite{epF5C2009}.
Note that there (NM) is considered only for $h$ being an element generated in the
current iteration step, that means $\idx(h) = i$. Reviewing the proof one easily
sees that the situation of $\idx(h)<i$ is just a special case already considered
by the proof: There the two 
signatures $u_f \sig(f)$ and $u_g \sig(g)$ refer to $\mathcal{M}$ and 
$\mathcal{N}$, where $\mathcal{M} \succ \mathcal{N}$. In our situation $u_h 
\sig(h)=\mathcal{N}$, and any principal syzygy $\omega$ with the above
mentioned properties can only decrease $\mathcal{N}$. The statement then follows by
the very same argumentation as in the proof given in~\cite{epF5C2009}.
\end{proof}

\begin{rem}\
\begin{enumerate}
\item Note that all signature-based algorithms have in common that they handle
their s-polynomials by increasing signature. This even holds for $\ff$, also
there the critical pairs are presorted by the degree of the pairs. Since $\ff$,
as presented in \cite{fF52002Corrected,epF5C2009}, works only with homogeneous input,
this does not interfere an ordering w.r.t. increasing signatures. By the
discussion in~\cite{epSig2011} $\ff$ can also be used
for inhomogeneous input by removing the presorting of critical
pairs by increasing degree.
\item $\ffc$ and $\ggv$ implement (NM) and (RW) quite differently. An
in-depth discussion on this topic is given in \cite{epSig2011}. These
differences are explained in the following, too, inasmuch as they influence the
ideas of this paper. 
\item\label{rem:nmrw} The crucial fact for the optimization presented in this paper is that whereas 
checking (NM) is quite easy and cheap speaking in a computational manner, searching
for possible elements $r$ with which we can check (RW) costs many more CPU cycles. 
\end{enumerate}
\end{rem}

\section{Computational overhead}
\label{sec:ffc}
One of the main problems of signature-based Gr\"obner basis algorithms is the
overhead generated by the following kind of data: 
\begin{enumerate}
\item From the point of view of the resulting Gr\"obner basis the elements are useless, that
means the corresponding leading monomials are superfluous.
\item For the correctness of the algorithm the very same elements are crucial:
They are essential for the correct detection of useless critical pairs w.r.t.
(NM) and (RW).
\end{enumerate}

\begin{rem}
This characteristic is unique to signature-based algorithms and cannot be found in
other Buchberger-style Gr\"obner basis algorithms.
It does not only give a penalty on the performance, but unfortunately also causes problems with theoretical aspects,for example regarding the termination of $\ff$, see~\cite{egpF52011} for more
details.
\end{rem}

Next we state the pseudo code of the main loop of an incremental signature-based
Gr\"obner basis algorithm in the vein of $\ff$, denoted $\siggb$.
Note that we denote the incrementally called subalgorithm $\incsig$.

\begin{algorithm}[h]
\caption{\label{alg:ff}$\siggb$, an incremental signature-based Gr\"obner basis algorithm in
  the vein of $\ff$}
\begin{algorithmic}[1]

\REQUIRE{$F_m$}
\ENSURE{$G$, a Gr\"obner basis for $I_m$}
\STATE{$G_1 \gets \left\{(e_1,f_1)\right\}$}
\FOR{$(i=2,\dotsc,m)$}
\STATE{$f_i \gets \reduce\left(f_i,\poly(G_{i-1})\right)$}\label{sigstdr:reduce}
\IF{$(f_i \neq 0)$}
\STATE{$G_i \gets \incsig\left(f_i,G_{i-1}\right)$}
\ELSE
\STATE{$G_i \gets G_{i-1}$}
\ENDIF
\ENDFOR
\RETURN{$\poly(G_m)$}
\end{algorithmic}
\end{algorithm}

Let us start the discussion on computational overhead looking at $\ff$ as presented
in~\cite{fF52002Corrected} first, so the following disussions refers to Algorithm~\ref{alg:ff}. 
The first drawback is the computation of {\it non-minimal} intermediate
Gr\"obner bases.
\begin{enumerate}
\item Due to the fact that the signatures of the labeled polynomials must be 
kept valid during the (sig-safe) reductions taking place, some leading term 
reductions do not take place immediately, but are postponed. These reductions, 
needed to ensure correctness of the algorithm, are computed when generating 
new critical pairs later on\footnote{In~\cite{fF52002Corrected} this kind of generation
of new critical pairs is not postponed to the end of the current reduction
step, but those are added to the pair set in place. These two ways of
handling such a situation are nearly equivalent and do not trigger any
difference for the overall computations, see~\cite{epSig2011}. The way it is described here makes
it easier to see how the computational overhead is produced.}.
Thus at the end we could have three polynomials $\poly(f)$, 
$\poly(g)$, and $\poly(h)$ in $\poly(G_i)$ such that
\begin{enumerate}
\item $\lm(g) \mid \lm(f)$, but the reduction $f-ctg$ has not taken place due 
to $t\sig(g) \succ \sig(f)$, for some $c \in \field$, $t \in \monset$ such
that $\lt(f) = ct\lt(g)$.
\item $h$ is the result of the later on generated and reduced s-polynomial
$\spol{g}{f}=ctg-f$, which is sig-safe due to swapping $ctg$ and $f$.
\end{enumerate}
In the end, we only need two out of these three elements for a Gr\"obner basis; 
in a minimal Gr\"obner basis we would discard $\poly(f)$. The problem is that for the 
correctness of the ongoing incremental step of $\ff$ the labeled polynomial $f$
as well as its addition to $G_{i}$ is essential\footnote{See~\cite{egpF52011} for
more details, also on termination issues caused by this behaviour of signature-based algorithms.}:
Without adding $f$ to $G_{i}$ the critical pair 
$(g,f)$ would not be generated at all, thus the element $h$, possibly needed 
for the correctness of the Gr\"obner basis in the end, would never be computed.  
So we are not able to remove $f$ during the actual iteration step. 
\end{enumerate}

Clearly, in the same vein the problem of non-reducedness of the Gr\"obner basis 
$\poly(G_i)$, in particular, missing tail-reductions, can be understood.

\begin{enumerate}[resume]
\item Since $\ff$, when reducing with elements generated in the ongoing
iteration step, processes top-reductions only, elements can enter $G_i$
whose polynomials have tails not reduced w.r.t. $\poly(G_i)$.
The main argument for not doing complete reductions in 
this situation is the requirement of sig-safeness: Comparing the signatures before each possible tail-reduction can lead to quite worse 
timings. On the other hand, from the point of view of the resulting Gr\"obner basis 
$\poly(G_i)$, which consists only of polynomial data, we do not need 
to take care of sig-safeness and can tail-reduce the elements in $\poly(G_i)$ as usual 
without any preprocessed signature comparison. This is way faster than implementing 
tail-reductions during the iteration step, although we have to use the non-tail-reduced 
elements during a whole iteration step.
\end{enumerate}

From the above discussion we get the following situation:
\begin{enumerate}
\item The computational overhead {\it during} an iteration step is prerequisite 
for the correctness of $\ff$.
\item The set of labeled polynomials $G_i$ returned {\it after} the 
$i$th iteration step is used as input for the $(i+1)$st iteration step,
including the signatures.
\end{enumerate}

In~\cite{stF5Rev2005} Stegers found a way optimizing at least the reduction
steps w.r.t. elements of previous iteration steps. There the fact is used that
$\ff$ does not need to look for the signatures, due to the definition of $\prec$
all such reducers have a smaller index, and thus, a smaller signature: His variant of $\ff$ computes
another set of polynomials $B_i$ after each iteration step, namely the reduced
Gr\"obner basis of $I_{i}$ which is computed out of $\poly(G_i)$. In the
following iteration step reductions w.r.t. elements computed in previous
iteration steps are done by $B_i$, not by $\poly(G_i)$.

In~\cite{epF5C2009} the variant $\ffc$ of $\ff$ is presented, which is based on
the idea of Stegers, but goes way further:
$\ffc$ is an optimized version of $\ff$ interreducing the intermediate Gr\"obner
basis $\poly(G_i)$ to $B_i$ and uses these {\it polynomial data} as starting point for the
next iteration step. At this point we can look at Algorithm~\ref{alg:ff}
from~\cite{epSig2011}, which illustrates one single iteration step of incremental signature-based
Gr\"obner basis algorithms: Let $\ell=\# B_i $, any element $b_j \in B_i$ gets a new signature
$e_j$, so that we receive elements $g_j = (e_j,b_j)$ in $G_{i+1}$ for $1\leq j
\leq \ell$. $f_{i+1}$ is then added to $G_{i+1}$ by adjusting the index,
  $g_{\ell+1}
=(e_{\ell+1},f_{i+1})$. On the one hand, proceeding this way the corresponding signatures of
reduced polynomials are guaranteed to be correct from the view of the algorithm.
On the other hand, all previously available criteria for detecting useless critical pairs w.r.t. to labeled
polynomials of index $\leq \ell$ by (RW) in the upcoming iteration step are removed. 
Thus the question, if for the benefit of having less computational overhead is payed dearly by
less efficient criteria checks in the following iteration steps, needs to be asked.

The idea to get at least some data for using (RW) on elements of previous
iteration steps in the following can be summarized in this way:
\begin{enumerate}
\item For any two elements $g_j,g_k \in G_{i+1}$ with polynomial parts in $B_i$,
$\spol{g_j}{g_k}$ reduces to zero w.r.t. $G_{i+1}$, that means, it already has a
standard representation w.r.t. $G_{i+1}$ (since $B_i$ is already a reduced
    Gr\"obner basis). Moreover, we can quite easily get the
corresponding signatures of $\spol{g_j}{g_k}$ due to the fact that the
labels of the generating elements, $e_j$ and $e_k$, are trivial. 
\item With this in mind, we can create for each index $2\leq j \leq \ell$ a list of
signatures generated out of $\sig\left(\spol{g_j}{g_k}\right)$ for $k<j$. Those
signatures can be used by (RW) detecting useless critical pairs in the next
iteration step.
\end{enumerate}

Clearly, one wants to omit such a recomputation of signatures during two
iteration steps as much as possible. Luckily it is shown in~\cite{epF5C2009} that this is not
a problem at all:

\begin{prop}
\label{prop:ffc}
Let $\spol{f}{g}$ be any s-polynomial considered during an iteration step of
$\ff$ with $\idx(g) < \idx(f)$. Assume that $\frac{\tau(f,g)}{\lm(g)} g$ would be detected either by
(NM) or (RW). Then $\spol{f}{g}$ is also discarded in $\ffc$.
\end{prop}

\begin{cor}
\label{cor:ffc}
For an s-polynomial in $\ffc$ it is enough to check a generator $f$ by (NM) resp.
(RW) if $f$ was computed during the current iteration
step.
\end{cor}

\begin{proof}
See Theorem~27 resp. Corollary~28 in~\cite{epF5C2009}.
\end{proof}

Thus it follows that we do not need to recompute any signature after
interreducing the intermediate Gr\"obner basis $\poly(G_i)$ for checks with (RW).

Let us add the above ideas in the pseudo code of Algorithm~\ref{alg:ffc}. We
highlight the new step of interreducing the intermediate Gr\"obner basis,
          differing from the description of Algorithm~\ref{alg:ff}. There are two main changes: 
\begin{enumerate}
\item $G_i$ now denotes already a set of polynomials in
$\rp$, as the labeled polynomials are newly generated at the beginning of each
incremental step.
\item Instead of $\incsig$ we use a new algorithm $\incsigc$ which takes a
reduced Gr\"obner basis $B_{i-1}$ as a second argument. Note that for $\incsigc$ we
refer the reader to Algorithm~1 in~\cite{epSig2011}.
\end{enumerate}

\begin{algorithm}[h]
\caption{\label{alg:ffc}$\siggb$ with reduced intermediate Gr\"obner bases}
\begin{algorithmic}[1]

\REQUIRE{$F_m$}
\ENSURE{$G$, a Gr\"obner basis for $I_m$}
\STATE{$G_1 \gets \{f_1\}$}
\FOR{$(i=2,\dotsc,m)$}
\STATE{\textcolor{newcolor}{$B_{i-1} \gets 
\redsb\left(G_{i-1}\right)$}}\label{sigstdr:red}
\STATE{$f_i \gets \reduce(f_i,B_{i-1})$}\label{sigstdr:reduce}
\IF{$(f_i \neq 0)$}
\STATE{$G_i \gets \incsigc(f_i,B_{i-1})$}
\ELSE
\STATE{$G_i \gets G_{i-1}$}
\ENDIF
\ENDFOR
\RETURN{$G_m$}
\end{algorithmic}
\end{algorithm}

\begin{rem}\
\begin{enumerate}
\item Also we do not state proofs for Proposition~\ref{prop:ffc} and
Corollary~\ref{cor:ffc} here, but only refer
to~\cite{epF5C2009}, it is important to note that correctness of
Proposition~\ref{prop:ffc} is based on the fact that $\ffc$ uses (RW) as presented
above. $\ggv$ uses a relaxed variant of (RW), see \cite{epSig2011} for more
details. Still, $\ggv$ uses Algorithm~\ref{alg:ffc} 
as an outer loop for its incremental computations.
\item In the following we can consider incremental
signature-based Gr\"obner basis algorithms in general, i.e., 
Algorithm~\ref{alg:ffc} can be seen as a wrapper for $\ffc$ and $\ggv$.
\end{enumerate}
\end{rem}

Having a common basis of our algorithms in question, let us see next how the
initial presentation of $\ggv$ improved the field of signature-based computations.

\section{Using reductions to zero}
\label{sec:ggv}
As described in~\cite{epSig2011} $\ggv$ can be seen a variant of $\ffc$ using a
way more relaxed version of (RW): $\ggv$ only checks if the corresponding
s-polynomials of two critical pairs have the same signature when adding the
pairs to the pair set. In this situation only one of these two pairs is kept, the
other one is discarded. We refer to~\cite{epSig2011} for more details.

Thus $\ggv$'s efficiency is mainly based on its optimized variant of (NM). The idea can be
explained quite easily: Whereas $\ffc$ uses only principal syzygies for (NM),
          since they are known beforehand and can be precomputed, $\ggv$ goes
          one step further: 

\begin{defi}
\label{def:ggvcrit}
During the $(i+1)$st iteration step of $\incsigc$ we define
\begin{align*}
S_{i+1} := \{ te_{\ell+1} \in \sigset \mid~&te_{\ell+1} \textrm{ signature of an
  s-polynomial}\\
    &\textrm{that reduced sig-safe to zero }\},
\end{align*}
where $\ell+1$ is the current index of the labeled polynomials.
\end{defi}

\begin{lem}[Improved (NM)]
\label{lem:ggv}
Let $\ell+1$ be the current index of the $(i+1)$st incremental step of $\ggv$
computing a Gr\"obner basis $\poly(G_{i+1})$ for $I_{i+1}$. $\spol{f}{g}$ has a standard
representation w.r.t. $G_{i+1}$ if there exists an element $\omega \in
\psyz\left(B_{i} \cup \{f_{i+1}\}\right) \cup S_{\ell+1}$ with $\lm(\omega) =
t_{\omega}e_{\ell+1}$ such that $t_{\omega} \mid u_h t_h$,
where $u_h \sig(h) = u_h t_h e_{\ell+1}$ for either $h=f$, $\idx(f)=\ell+1$ or
$h=g$, $\idx(g) = \ell+1$. 
\end{lem}

\begin{proof}
See Proposition~16 and Lemma~17 in~\cite{epSig2011}.
\end{proof}

\begin{rem}\
\begin{enumerate}
\item Switching from $\psyz(F_{i+1})$ to $\psyz\left(B_{i} \cup \{f_{i+1}\}\right)$
is not a problem at all since $\langle F_{i+1}\rangle = \langle B_{i} \cup
\{f_{i+1}\}\rangle$ as $B_{i}$ is the reduced Gr\"obner basis of $I_{i}$.
\item Note that the restriction to check elements of current index only is influenced
by the discussion in Section~\ref{sec:ffc}. Whereas we know that
Corollary~\ref{cor:ffc} ensures that $\ffc$ does not lose any useful information
for rejecting useless critical pairs due to its aggressive implementation of (RW), 
this does not hold for $\ggv$.
So for $\ggv$ it is possible that removing the signatures of the intermediate
Gr\"obner bases can lead to the situation that less critical pairs are rendered
useless by (NM).
\end{enumerate}
\end{rem}

Clearly, one can easily use (NM) as given in Lemma~\ref{lem:ggv} in $\ffc$
instead of the plain one stated in Lemma~\ref{lem:criteria}, its correctness
does not depend on the implementation of (RW).

\begin{defi}
We denote the algorithm $\ffc$ with (NM) implemented as in Lemma~\ref{lem:ggv}
by $\ffe$.\footnote{The ``A'' stands for ``actively using zero reductions''.}
\end{defi}

The question that comes to one's mind is the following: Why would it be of 
any benefit to switch from $\ffc$ to $\ffe$? Whereas some of the signatures 
of zero reductions are used in $\ffc$ in (RW), not all of them can be used
in each situation due to the restriction that $\sig(r) \succ \sig(h)$. It is 
already mentioned in~\cite{epSig2011} that 
$\ffe$ is way faster than $\ffc$ for non-regular input: $\ffc$ computes 
way more zero reductions than $\ffe$, but cannot use it actively. Moreover, 
even if a corresponding (RW) detection happens in $\ffc$, testing by (NM) in
$\ffe$ is a lot faster as already discussed in Remark~\ref{rem:nmrw}. We see
in Section~\ref{sec:exp} that this has a huge impact not only on the timings, 
but also on the number of reduction steps taking place, due to the fact that
(NM) and (RW) are used in $\ffc$ resp. $\ffe$ on possible reducers, too.

\section{Combining ideas}
\label{sec:idea}
So what is our contribution in this paper besides giving an overview of already known
optimizations? Until now, the presented ideas
of interreducing intermediate Gr\"obner bases and using zero reductions actively
in (NM) are used without any direct connection:
\begin{enumerate}
\item The Gr\"obner bases are interreduced {\it between two iteration steps}.
This has an effect on the labeled polynomials computed in the previous iteration
steps.
\item Zero reductions are used actively {\it in a single iteration step only}.
This has an impact on current index labeled polynomials only.
\end{enumerate}

Of course, interreducing the intermediate Gr\"obner basis $\poly(G_i)$ to $B_i$has an influence on
the upcoming iteration step inasmuch as less critical pairs are considered and
reductions w.r.t. $B_i$ are more efficient. Besides this we cannot assume
to receive any deeper impact on the $(i+1)$st iteration step. On the other hand,
   it would be quite nice to use (NM) not only on current index labeled
   polynomials, but also on those coming from $B_i$. For this one could just
precompute $\psyz(B_i)$ and check the corresponding lower index generators of
critical pairs using $\psyz(B_i)$ in (NM). By Lemma~\ref{lem:criteria} this
would be a correct optimization. The crucial point is that we can do even
better:

At the moment we interreduce the intermediate Gr\"obner basis $\poly(G_i)$ of $I_i$
let us try to get some more resp. better signatures for checking (NM): We have
already seen in Section~\ref{sec:ffc} how to recompute some signatures for
checking lower index labeled polynomials in the upcoming iteration step: We
just compute and keep the signatures of the corresponding s-polynomials 
$\spol{g_j}{g_k}$ for $\poly(g_j),\poly(g_k) \in B_{i}$. There we could reject
their computations as we have proven that $\ffc$ would discard any s-polynomial
detected with those signatures by (RW) anyway. Moreover, we know that all those signatures 
correspond to s-polynomials that already reduce to zero w.r.t. $B_{i}$. Thus, with
the idea of Section~\ref{sec:ggv} in mind, we can actively use those theoretical zero
reductions actively for lower index labeled polynomials.

\begin{defi}
Assume that $\poly(G_{i})$ is reduced to 
$B_{i}=\left\{b_1,\dotsc,b_{\ell_i}\right\}$
after the $i$th iteration step of $\siggb$.
Then we define
\[S_i := \left\{\frac{\tau(b_{\ell_i},b_k)}{\lm(b_{\ell_i})} e_{\ell_i} \mid 
1\leq k < \ell_i
\right\}.\]
\end{defi}

\begin{thm}[Strengthening (NM)]
\label{lem:nmopt}
Assuming the $(i+1)$st incremental step of a signature-based algorithm computing 
a Gr\"obner basis $\poly(G_{i+1})$ for $I_{i+1}$, let $\ell+1$ be the current index of
labeled polynomials. $\spol{f}{g}$ has a standard
representation w.r.t. $G_{i+1}$ if there exists an element
\[\omega \in \psyz\left(B_{i} \cup \{f_{i+1}\}\right) \cup S_{i+1} \cup S_{i} 
\cup \dotsc \cup
S_{2}\]
with $\lm(\omega) = t_{\omega}e_{j_h}$ such that $t_{\omega} \mid u_h t_h$,
where $u_h \sig(h) = u_h t_h e_{j_h}$ for either $h=f$ or $h=g$. 
\end{thm}

\begin{proof}
Clearly, the elements $\omega$ in $\psyz\left(B_{i} \cup \{f_{i+1}\}\right)$ can be divided
in two types:
\begin{enumerate}
\item $\lm(\omega) = t e_{\ell+1}$, and
\item $\lm(\omega) = t e_j$ for $2 \leq j < \ell+1$.
\end{enumerate}
Those elements of Type (1) from $\psyz\left(B_{i} \cup \{f_{i+1}\}\right)$
together with those of $S_{i+1}$ (see Definition~\ref{def:ggvcrit}) are of
the type $t e_{\ell+1}$, and thus do only detect generators of the s-polynomial 
in
question which are computed during the current iteration step. The correctness
of this statement follows from Lemma~\ref{lem:ggv}.

Next we consider $S_j$ for $2 \leq j \leq i$: Note that each element in
such an $S_j$ is of type $t e_{\ell_j}$, that means those elements can only 
detect
generators of the s-polynomial in question of index $\ell_j$. Thus it is enough 
to
consider one such $S_j$ for a fixed $j$. Let $te_{\ell_j} \in S_j$, then $t =
\frac{\tau(b_{\ell_j},b_k)}{\lm(b_{\ell_j})}$ for some $1\leq k < \ell_j$. Since 
$B_j$ with $\# B_j = \ell_j$ is a reduced
Gr\"obner basis we know that $\spol{b_{\ell_j}}{b_k}$ reduces to zero
w.r.t. $B_j$. Moreover, note that due to the incremental structure of $\siggb$ 
and the
  interreduction of $\poly(G_j)$ between each two iteration steps the sequence
$(b_1,\dotsc,b_{\ell_j})$ of elements of $B_{j}$ can be ordered in such a
    way that this zero reduction corresponds to a syzygy $\omega$ with
    $\lm(\omega) = t e_{\ell_j}$. Any other possible syzygy detecting a 
    generator of index $\ell_j$ is of Type (2) from $\psyz\left(B_{i} \cup 
        \{f_{i+1}\}\right)$.  By
Lemma~\ref{lem:criteria}~(\ref{nmcriterion}) the
statement holds. 
\end{proof}

\begin{rem} \
\begin{enumerate}
\item Note that the sets $S_i$ in Theorem~\ref{lem:nmopt} are computed 
recursively after the $i$th iteration step of $\siggb$. There is no connection 
between $S_i$ and $S_{i-1}$ for any $i>3$. It is also crucial to only take the 
element of highest index $\ell_i$ from $B_i$ into account when computing $S_i$, 
        we cannot ensure that $\spol{g_j}{g_k}$ reduces sig-safe to zero w.r.t.  
        $G_i$ if $j \neq \ell_i \neq k$. There a not sig-safe reduction could be 
        possible to achieve the zero reduction of $\spol{g_j}{g_k}$, a problem 
        that cannot happen if either $j$ or $k$ is equal to $\ell_i$.  \item The 
        main optimization compared to $\ffc$ is to use $S_{i}$ up to $S_2$ not 
        in (RW), as it is described in Section~\ref{sec:ffc}, but in (NM). 
        Compared to $\ggv$'s variant of (NM) lots of new checks are added that 
        detect way more useless critical pairs as we see in 
        Section~\ref{sec:exp}.

Note that this problem has not been taken into account in \cite{epF5C2009} where 
such signatures are suggested to be used in (RW).
\end{enumerate}
\end{rem}

One could think that the presented strengthening of (NM) is rather equivalent to
the initial presentation in Lemma~\ref{lem:criteria}, but this is not the fact.
The variant presented here is way more aggressive in finding useless critical
pairs.

\begin{cor}
In Theorem~\ref{lem:nmopt} elements of type $\lm(s) = te_j \in \psyz\left(B_{i}
    \cup \{f_{i+1}\}\right)$ where $j = \# B_\ell$ for some $2\leq \ell \leq i$ 
is need not be
considered at all.
\end{cor}

\begin{proof}
Assume such an $s \in \psyz\left(B_{i} \cup \{f_{i+1}\}\right)$ with $\lm(s) = t
e_j$, $j = \# B_\ell$ for some fixed $\ell$. Then $s= b_k e_j - b_j e_k$ for 
some $k<j$. Let $k$ be fixed. In $S_\ell$ there exists some $u e_j$ with $u =
\frac{\tau(b_j,b_k)}{\lm(b_j)}$. It follows that $u \mid t$.
\end{proof}

The idea is now to replace the implementations of (NM) in $\ffc$, $\ffe$, and $\ggv$
, respectively, by the version given in Theorem~\ref{lem:nmopt}.

\begin{defi}
We denote the algorithms $\ffc$, $\ffe$, and $\ggv$ with (NM) implemented as in
Theorem~\ref{lem:nmopt} by $\ffcc$, $\ffec$, and $\ggvc$,
  respectively.\footnote{The ``i'' stands for ``intermediate incremental 
    optimization''.}
\end{defi}

Algorithm~\ref{alg:ffca} illustrates the main wrapper for an incremental
signature-based Gr\"obner basis algorithm based on the idea presented in this
section.

\begin{algorithm}[h]
\caption{\label{alg:ffca}$\siggb$ with reduced intermediate Gr\"obner bases and optimized (NM)
  Criterion}
\begin{algorithmic}[1]

\REQUIRE{$F_m$}
\ENSURE{$G$, a Gr\"obner basis for $I_m$}
\STATE{\textcolor{newcolor}{$S \gets []$}, $G_1 \gets \{f_1\}$}
\FOR{$(i=2,\dotsc,m)$}
\STATE{\textcolor{newcolor}{$B_{i-1} \gets 
  \redsb\left(G_{i-1}\right)$}}\label{ffca:red}
\STATE{$f_i \gets \reduce(f_i,B_{i-1})$}\label{ffca:reduce}
%\STATE{\textcolor{newcolor}{$\ell \gets \# B_{i-1}$}}
\IF{$(f_i \neq 0)$}
\STATE{\textcolor{newcolor}{$j \gets \# B_{i-1}$}}\label{ffca:start}
\FOR{\textcolor{newcolor}{$\left(k=1,\dotsc,j-1\right)$}}
\STATE{\textcolor{newcolor}{$S_{i-1,k} \gets 
  \frac{\tau(b_j,b_k)}{\lm(b_j)}e_j$}}\label{ffca:stop}
\ENDFOR
\STATE{$G_i \gets \incsigc(f_i,B_{i-1})$}
\ELSE
\STATE{$G_i \gets G_{i-1}$}
\ENDIF
\ENDFOR
\RETURN{$G_m$}
\end{algorithmic}
\end{algorithm}

\begin{example}
\label{ex:problem}
Let $p_1 = yz+2$, $p_2 = xy+\frac 1 3 xz+\frac 2 3$, $p_3 = xz^2-6x+2z$ be three 
polynomials in $\mathbb{Q}[x,y,z]$. We want to compute a Gr\"obner basis for the 
ideal $I=\langle p_1, p_2, p_3\rangle$ w.r.t. the degree reverse lexicographical 
ordering with $x > y > z$ using Algorithm~\ref{alg:ffca} using $\siglt$ as 
ordering on the signatures.

Due to the incremental structure of the algorithm we start with the computation 
of a Gr\"obner basis $G_2$ of $\langle p_1,p_2\rangle$.
After initializing $f_1 := (e_1,p_1)$ and $f_2:=(e_2,p_2)$ we construct
the s-polynomial of $f_1$ and $f_2$ is
\[\spol {f_2}{f_1} = (ze_2, zp_2 - xp_1),\]
which does not have a standard representation w.r.t. $G_2$ at the moment of its 
creation. Nevertheless, a new element is added to $G_2$, namely
\[f_3 := \left(ze_2,\frac 1 3 xz^2-2x+\frac 2 3 z\right).\] 

Now we look at the s-polynomials
\begin{align*}
\spol{f_3}{f_1} &= \left(yz e_2,3y \poly(f_3) - xz \poly(f_1)\right),\\
\spol{f_3}{f_2} &= \left(yz e_2,3y \poly(f_3) - z^2 \poly(f_2)\right).
\end{align*}

It does not make any difference which of these two s-polynomials we compute: 
$\ff$ would remove the later one by its implementation of (RW), whereas $\ggv$ 
would only store one of the two corresponding critical pairs in the beginning.  
W.l.o.g. we assume the reduction of $\spol{f_3}{f_1}$, the situation for 
considering $\spol{f_3}{f_2}$ is similar and behaves the very same way:
\begin{align*}
\spol{f_3}{f_1} &= \left(yz e_2, -6xy-2xz+2yz\right)
\end{align*}
Further sig-safe reductions with $6f_2$ and $2f_1$ lead to a zero reduction, 
        i.e. we can add a new rule, namely $yz e_2$, to the set $S_2$ as 
        explained in Section~\ref{sec:ggv}.\footnote{Latest at this point 
          $\spol{f_3}{f_2}$ would be removed by the rule $yz e_2 \in S_2$.} 
        
Since there is no further s-polynomial left, $\siggb$ finishes this iteration 
step with
\[G_2 = \left\{ yz+2, xy+\frac 1 3 xz + \frac 2 3, \frac 1 3 xz^2-2x+\frac 2 3 
z\right\}.\]
We see that $G_2$ is already the reduced Gr\"obner basis $B_2$ of $\langle 
p_1,p_2\rangle$, so Line~\ref{ffca:red} of $\siggb$ does not change anything.
Since $G_i$ and $B_i$ do not coincide inbetween most iteration steps, we need to 
remove $S_2$ completely, since rules stored in there might not be correct any 
more.

Let us have a closer look at how the new rules would be computed: First of all, 
    any element in $G_2$ corresponds to a module generator of $\rp^3$, that 
    means we can think of three labeled polynomials $f_i:=(e_i,b_i)$ for $i\in 
    \{1,2,3\}$.
During the lines~\ref{ffca:start} --~\ref{ffca:stop} the signatures 
corresponding to $\spol{f_3}{f_1}$ and $\spol{f_3}{f_2}$ are added to $S_2$ 
respectively:
The first one gives the rule $ye_3$, the second one also; thus we have $S_2 = 
\{ye_3\}$.

Next we see that $p_3$ reduces to zero w.r.t. $B_2$ in Line~\ref{ffca:reduce}.  

Thus the algorithm does not enter another iteration step,
  but terminates with the correct Gr\"obner basis $G = B_2$.
\end{example}

\begin{rem} \
\begin{enumerate}
\item Of course as stated in the pseudo code, $\siggb$ would reduce the next 
generator $p_3$ of the input ideal w.r.t. intermediate reduced Gr\"obner basis 
$B_2$, {\em before} it would recompute the syzygy list $S_2$ in the above 
example. For the sake of explaining how the recomputation of such a rules list 
works, using a rather small example, we choose a complete discussion over 
efficiency in Example~\ref{ex:problem}
\item Note that it would be wrong to also add the signature of $\spol{f_2}{f_1}$ 
to $S_2$ once we have interreduced $G_2$ to $B_2$. Of course, $\spol{b_2}{b_1}$ 
has a standard representation w.r.t.  $B_2$, namely
\[\spol{b_2}{b_1} = b_3,\] but this does not lead to a standard representation 
of $\spol{f_2}{f_1}$ due to the fact that $\sig(f_3)= e_3 \succ 
\sig\left(\spol{f_2}{f_1}\right)$. That is the one big drawback of interreducing 
the intermediate Gr\"obner bases in incremental signature-based algorithms.  
Nevertheless, the speed up due to handling way less elements in $B_i$ compared 
to $G_i$ more than compensates this as shown in \cite{epF5C2009}.
\end{enumerate}
\end{rem}

\section{Experimental results}
\label{sec:exp}
We compare timings, the number of zero reductions, and the number of overall
reduction steps of the different algorithms presented in this paper.
To give a faithful comparison, we use a further developed version of the
implementation we have done for~\cite{epSig2011}: This is an implementation of a generic
signature-based Gr\"obner basis algorithm in the kernel of a developer version of \textsc{Singular}
3-1-4. Based on this version we implemented $\ggv$, $\ggvc$, $\ffc$, $\ffcc$,
$\ffe$, and $\ffec$ by plugging in the different variants and usages of the
criteria (NM) and (RW). There are no optimizations which
could prefer any of the specific algorithms, so that the difference in the 
implementation between two of the above mentioned algorithms is not more than 
300 lines of code; compared to approximately 3,500 lines of code overall this is 
    negligible. All share the very same data structures and use the very same 
    (sig-safe) reduction routines. So the differences shown in 
    Tables~\ref{tab:ch5:time},~\ref{tab:ch5:totalRed}, and~\ref{tab:ch5:zero} 
    really come from the various optimizations of the criteria mentioned in 
    Sections~\ref{sec:ffc} --~\ref{sec:idea}.

Of course, to ensure such an accurate comparison of various different variants 
of signature-based algorithms has a drawback in the overall performance of the 
algorithms. Since we are interested in impact of the improvements explained in 
this paper it is justified to take such an approach. Clearly, implementing a 
highly optimized $\ffec$ whithout any restrictions due to sharing data 
structures and procedures with an $\ggv$ can lead to a way better performance.  
It is not in focus of this paper to present the fastest implementation of such 
kind of algorithms, but to present practical benefits of the presented 
optimizations, focusing on the fact that all variants of incremental 
signature-based Gr\"obner basis algorithms take an advantage out of them.

The source code is publicly available in the branch {\tt f5cc}\footnote{The 
  results presented here are done with the commit key {\tt 
    5c4dc1134a4ab630faab994dbe93d3013b4ccc7e}.} at

\begin{center}
\url{https://github.com/ederc/Sources}.
\end{center}
We computed the examples on a computer with the following specifications: 
\begin{itemize}
\item 2.6.31--gentoo--r6 GNU/Linux 64--bit operating system, 
\item INTEL\textregistered~XEON\textregistered~X5460~@~3.16GHz processor, 
\item 64 GB of RAM, and 
\item 120 GB of swap space.
\end{itemize}

Due to the fact that we are comparing 6 algorithms we colorized the results
presented in the respective tables for better readability, see Figure~\ref{fig:colortables}.

\begin{figure}
\begin{center}
\includegraphics[width=0.8\textwidth]{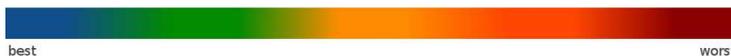}
\end{center}
\caption{\label{fig:colortables}Meanings of the colors in the respective tables}
\end{figure}

Our naming convention for examples specifies that ``{\tt -h}'' denotes
the homogenized variant of the corresponding benchmark. All examples are
computed over a field of characteristic $32,003$ w.r.t. the degree reverse
lexicographical ordering.

First of all let us have a closer look at $\ggv$, $\ffc$, and $\ffe$. In
Table~\ref{tab:ch5:time}, we see
that whereas $\ffc$ is faster than $\ggv$ in nearly all example sets, the ones
which lead to a high number of zero reductions in $\ffc$ are computed way faster
by $\ggv$. This is based on the fact that $\ggv$ actively uses such zero
reductions with adding new checks for (NM), whereas $\ffc$ only partially
includes those signatures in its implementation of (RW). 

Comparing $\ffe$ and
$\ffc$ we see that $\ffe$ is not only way faster than $\ffc$ in such highly
non-regular examples like {\tt Eco-X}\footnote{Note that $\ffc$ and $\ffcc$ cannot compute
  {\tt Eco-11-h}.}, but also that $\ffe$ is faster in other
systems like {\tt Cyclic-8}. Note that $\ffe$ is in all examples faster and 
computes less reductions than $\ggv$. 

The ideas of Section~\ref{sec:idea} help $\ggvc$ to compute much less
reduction steps and discard way more useless critical pairs than $\ggv$ does.  
This comes from the fact that $\ggv$ does not implement any rewritable criterion 
besides its choice of keeping only 1 critical pair per signature. In most 
examples $\ggvc$ executes only half as much reduction steps as $\ggv$, in some 
examples like {\tt F-855} it even alleviates to $15\%$. 

By our discussion in Section~\ref{sec:ffc} it is not a surprise that there is no 
change in the corresponding numbers of reduction steps and zero reductions for 
$\ffcc$ resp. $\ffec$ compared to the ones of $\ffc$ and $\ffe$.  Still the 
timings improve greatly which is based on the following facts:
\begin{enumerate}
\item Altough the rules added to $S_i$ in Lines~\ref{ffca:start} 
--~\ref{ffca:stop} are also checked by $\ff$'s (RW) implementation, checking 
(RW) costs more time than checking (NM) due to checking if the one is still in 
the area of correct (RW) rules.
\item To have all possible (RW) rules available, (RW) must be checked directly 
before the reduction step of the corresponding critical pair starts. At this 
point the critical pair can be stored for a long time, using memory and making 
the list of critical pairs longer. In $\ffcc$ resp. $\ffec$ useless critical 
pairs can be found directly when the algorithm tries to create them. Thus they 
are not kept for a long time, keeping lists shorter, which does not only save 
memory, but also speeds up insertion of other critical pairs to the list.
\end{enumerate}
All three variants implementing the idea of Section~\ref{sec:idea} are, often 
multple times, faster in all examples, besides {\tt Katsura-X}; for them we 
already know that the usual implementation of (NM), that means considering the 
principal syzygies, is already optimal. Thus all the ideas presented in this 
paper only add some small overhead in the computations of {\tt Katsura-X} which 
does not affect the performance in a beneficial way. 

\begin{center}
\begin{table}
\begin{centering}
  	\begin{tabular}{|c|D{.}{.}{3}|D{.}{.}{3}|D{.}{.}{3}|D{.}{.}{3}|D{.}{.}{3}|D{.}{.}{3}|}
		\hline
  		Test case & \multicolumn{1}{c|}{$\ggv$} & \multicolumn{1}{c|}{$\ggvc$}   & \multicolumn{1}{c|}{$\ffc$} & \multicolumn{1}{c|}{$\ffcc$} &   \multicolumn{1}{c|}{$\ffe$}   & \multicolumn{1}{c|}{$\ffec$}\\ 
		\hline
		\hline
    {\tt Cyclic-7-h} & {\color{6}25}.{\color{6}714} & 
{\color{5}13}.{\color{5}441} & {\color{3}5}.{\color{3}369} & 
{\color{2}5}.{\color{2}255} & {\color{4}5}.{\color{4}790} & 
{\color{1}5}.{\color{1}180}\\ \hline
    {\tt Cyclic-7} & {\color{6}24}.{\color{6}613} & {\color{5}12}.{\color{5}782} 
    & {\color{4}5}.{\color{4}660} & {\color{3}5}.{\color{3}310} & 
{\color{2}5}.{\color{2}216} & {\color{1}4}.{\color{1}758}\\ \hline
    {\tt Cyclic-8-h} & {\color{6}13,273}.{\color{6}487} & 
{\color{4}5,710}.{\color{4}221} & {\color{5}6,881}.{\color{5}895} & 
{\color{2}3,689}.{\color{2}003} & {\color{3}4,970}.{\color{3}714} & 
{\color{1}1,907}.{\color{1}076}\\ \hline
    {\tt Cyclic-8} & {\color{6}12,386}.{\color{6}925} & 
{\color{4}5,362}.{\color{4}307} & {\color{5}6,606}.{\color{5}116} & 
{\color{2}3,482}.{\color{2}177} & {\color{3}4,814}.{\color{3}173} & 
{\color{1}1,812}.{\color{1}321}\\ \hline
    {\tt Eco-9-h} & {\color{4}10}.{\color{4}050} & {\color{3}6}.{\color{3}741} & 
{\color{5}58}.{\color{5}935} & {\color{6}59}.{\color{6}182} & 
{\color{2}6}.{\color{2}320} & {\color{1}5}.{\color{1}550}\\ \hline
    {\tt Eco-9} & {\color{6}18}.{\color{6}755} & {\color{1}10}.{\color{1}012} & 
{\color{2}12}.{\color{2}899} & {\color{4}12}.{\color{4}994} & 
{\color{5}13}.{\color{5}425} & {\color{2}12}.{\color{2}899}\\ \hline
    {\tt Eco-10-h} & {\color{4}278}.{\color{4}663} & 
{\color{2}158}.{\color{2}167} & {\color{6}2,472}.{\color{6}763} & 
{\color{5}1,974}.{\color{5}163} & {\color{3}188}.{\color{3}683} & 
{\color{1}133}.{\color{1}183}\\ \hline
    {\tt Eco-10} & {\color{6}1,041}.{\color{6}896} & 
{\color{3}497}.{\color{3}025} & {\color{1}491}.{\color{1}914} & 
{\color{5}518}.{\color{5}346} & {\color{2}492}.{\color{2}345} & 
{\color{4}497}.{\color{4}811}\\ \hline
    {\tt Eco-11-h} & {\color{4}10,169}.{\color{4}897} &
{\color{2}5,136}.{\color{2}465} & {\color{5}-} & {\color{5}-} & 
{\color{3}7,893}.{\color{3}788} & {\color{1}4,815}.{\color{1}971}\\ \hline
    {\tt F-744-h} & {\color{4}33}.{\color{4}244} & {\color{3}25}.{\color{3}724} 
    & {\color{5}35}.{\color{5}324} & {\color{6}36}.{\color{6}740} & 
{\color{1}19}.{\color{1}865} & {\color{1}19}.{\color{1}865}\\ \hline
    {\tt F-744} & {\color{6}27}.{\color{6}312} & {\color{5}20}.{\color{5}469} & 
{\color{1}8}.{\color{1}198} & {\color{4}9}.{\color{4}267} & 
{\color{3}8}.{\color{3}458} & {\color{1}8}.{\color{1}198}\\ \hline
    {\tt F-855-h} & {\color{4}1,246}.{\color{4}744} & 
{\color{1}290}.{\color{1}856} & {\color{6}2,948}.{\color{6}138} & 
{\color{5}2,381}.{\color{5}813} & {\color{3}600}.{\color{3}001} & 
{\color{2}422}.{\color{2}219}\\ \hline
    {\tt F-855} & {\color{6}971}.{\color{6}491} & {\color{5}131}.{\color{5}134} 
    & {\color{2}83}.{\color{2}185} & {\color{4}86}.{\color{4}343} & 
{\color{3}86}.{\color{3}558} & {\color{1}84}.{\color{1}020}\\ \hline
    {\tt Katsura-10-h} & {\color{1}4}.{\color{1}186} & 
{\color{2}4}.{\color{2}193} & {\color{3}4}.{\color{3}213} & 
{\color{6}4}.{\color{6}491} & {\color{4}4}.{\color{4}248} & 
{\color{5}4}.{\color{5}256}\\ \hline
    {\tt Katsura-10} & {\color{1}4}.{\color{1}150} & {\color{2}4}.{\color{2}183} 
    & {\color{3}4}.{\color{3}187} & {\color{6}4}.{\color{6}217} & 
{\color{5}4}.{\color{5}227} & {\color{4}4}.{\color{4}192}\\ \hline
    {\tt Katsura-11-h} & {\color{4}59}.{\color{4}004} & 
{\color{5}59}.{\color{5}871} & {\color{3}58}.{\color{3}689} & 
{\color{6}61}.{\color{6}496} & {\color{1}58}.{\color{1}411} & 
{\color{2}58}.{\color{2}673}\\ \hline
    {\tt Katsura-11} & {\color{5}53}.{\color{5}894} & 
{\color{4}53}.{\color{4}855} & {\color{2}53}.{\color{2}464} & 
{\color{6}56}.{\color{6}122} & {\color{3}53}.{\color{3}984} & 
{\color{1}53}.{\color{1}118}\\ \hline
    {\tt Gonnet-83-h} & {\color{4}12}.{\color{4}165} & 
{\color{3}10}.{\color{3}617} & {\color{6}126}.{\color{6}173} & 
{\color{5}25}.{\color{5}963} & {\color{2}9}.{\color{2}811} & 
{\color{1}8}.{\color{1}761}\\ \hline
    {\tt Schrans-Troost-h} & {\color{6}4}.{\color{6}393} & 
{\color{5}4}.{\color{5}250} & {\color{1}2}.{\color{1}970} & 
{\color{4}3}.{\color{4}498} & {\color{2}3}.{\color{2}087} & 
{\color{1}2}.{\color{1}970}\\ \hline
  \end{tabular}
	\par
\end{centering}

\caption{Time needed to compute a Gr\"obner basis of the respective test case, given in seconds.}
\label{tab:ch5:time}
\end{table}
\end{center}

\begin{table}
\begin{centering}
  	\begin{tabular}{|c|c|c|c|c|}
		\hline
  		Test case & $\ggv$ & $\ggvc$   & $\ffc$,$\ffcc$ & $\ffe$,$\ffec$\\   
		\hline
		\hline
		{\tt Cyclic-7-h} & {\color{6}1,750,989} & {\color{5}625,815} & {\color{3}100,569} & {\color{1}83,880}\\ 
		\hline
		{\tt Cyclic-7} & {\color{6}1,750,989} & {\color{5}625,815} & {\color{3}100,569} & {\color{1}83,880}\\ 
		\hline
		{\tt Cyclic-8-h} & {\color{6}113,833,183} & {\color{5}44,663,466} & {\color{3}14,823,873} & {\color{1}3,403,874}\\ 
		\hline
		{\tt Cyclic-8} & {\color{6}113,833,183} & {\color{5}44,663,466} & {\color{3}14,823,873} & {\color{1}3,403,874}\\ 
		\hline
		{\tt Eco-9-h} & {\color{4}409,880} & {\color{3}238,841} &
{\color{5}1,996,849} & {\color{1}136,842}\\ 
		\hline
		{\tt Eco-9} & {\color{6}551,837} & {\color{5}310,745} & {\color{1}247,434} & {\color{1}247,434}\\ 
		\hline
		{\tt Eco-10-h} & {\color{4}3,760,244} & {\color{3}1,996,573} & {\color{5}19,755,560} & {\color{1}1,019,439}\\ 
		\hline
		{\tt Eco-10} & {\color{6}6,853,713} & {\color{5}3,352,474} & {\color{1}2,384,889} & {\color{1}2,384,889}\\ 
		\hline
		{\tt Eco-11-h} & {\color{4}33,562,613} & {\color{3}16,695,766} & {\color{5}-} & {\color{1}7,374,779}\\ 
		\hline
		{\tt F-744-h} & {\color{6}1,082,448} & {\color{3}693,630} & {\color{4}789,072} & {\color{1}435,869}\\ 
		\hline
		{\tt F-744} & {\color{6}473,838} & {\color{5}285,402} & {\color{1}179,100} & {\color{1}179,100}\\ 
		\hline
		{\tt F-855-h} & {\color{6}23,097,574} & {\color{3}4,407,938} & {\color{4}12,294,951} & {\color{1}2,633,666}\\ 
		\hline
		{\tt F-855} & {\color{6}7,976,163} & {\color{5}1,772,726} & {\color{1}835,718} & {\color{1}835,718}\\ 
		\hline
    {\tt Katsura-10-h} & {\color{5}18,955} & {\color{5}18,955} & 
{\color{1}18,343} & {\color{1}18,343}\\ \hline
		{\tt Katsura-10} & {\color{5}18,955} & {\color{5}18,955} & {\color{1}18,343} & {\color{1}18,343}\\ 
		\hline
		{\tt Katsura-11-h} & {\color{5}65,991} & {\color{5}65,991} & {\color{1}63,194} & {\color{1}63,194}\\ 
		\hline
		{\tt Katsura-11} & {\color{5}65,991} & {\color{5}65,991} & {\color{1}63,194} & {\color{1}63,194}\\ 
		\hline
    {\tt Gonnet-83-h} & {\color{4}113,609} & {\color{1}93,137} & 
{\color{5}278,419} & {\color{1}93,137}\\ \hline
		{\tt Schrans-Troost-h} & {\color{6}19,132} & {\color{5}18,352} & {\color{1}14,010} & {\color{1}14,010}\\ 
		\hline
	\end{tabular}
	\par
\end{centering}

\caption{Number of all reduction steps throughout the computations of the algorithms.}
\label{tab:ch5:totalRed}
\end{table}
\begin{table}
\begin{centering}
  	\begin{tabular}{|c|c|c|c|}
		\hline
  		Test case & $\ggv$,$\ggvc$   & $\ffc$,$\ffcc$ & $\ffe$,$\ffec$\\   
		\hline
		\hline
		{\tt Cyclic-7-h} & {\color{1}36} & {\color{5}76} & {\color{1}36}\\ 
		\hline
		{\tt Cyclic-7} & {\color{1}36} & {\color{5}76} & {\color{1}36}\\ 
		\hline
		{\tt Cyclic-8-h} & {\color{1}244} & {\color{5}1,540} & {\color{1}244}\\ 
		\hline
		{\tt Cyclic-8} & {\color{1}244} & {\color{5}1,540} & {\color{1}244}\\ 
		\hline
		{\tt Eco-9-h} & {\color{1}120} & {\color{5}929} & {\color{1}120}\\ 
		\hline
		{\tt Eco-9} & {\color{1}0} & {\color{1}0} & {\color{1}0}\\ 
		\hline
		{\tt Eco-10-h} & {\color{1}247} & {\color{5}2,544} & {\color{1}247}\\ 
		\hline
		{\tt Eco-10} & {\color{1}0} & {\color{1}0} & {\color{1}0}\\ 
		\hline
		{\tt Eco-11-h} & {\color{1}502} & {\color{5}-} & {\color{1}502}\\ 
		\hline
		{\tt F-744-h} & {\color{1}323} & {\color{5}498} & {\color{1}323}\\ 
		\hline
		{\tt F-744} & {\color{1}0} & {\color{1}0} & {\color{1}0}\\ 
		\hline
		{\tt F-855-h} & {\color{1}835} & {\color{5}2,829} & {\color{1}835}\\ 
		\hline
		{\tt F-855} & {\color{1}0} & {\color{1}0} & {\color{1}0}\\ 
		\hline
    {\tt Katsura-10-h} & {\color{1}0} & {\color{1}0} & {\color{1}0}\\ \hline
		{\tt Katsura-10} & {\color{1}0} & {\color{1}0} & {\color{1}0}\\ 
		\hline
		{\tt Katsura-11-h} & {\color{1}0} & {\color{1}0} & {\color{1}0}\\ 
		\hline
		{\tt Katsura-11} & {\color{1}0} & {\color{1}0} & {\color{1}0}\\ 
		\hline
    {\tt Gonnet-83-h} & {\color{1}2,005} & {\color{5}8,129} & {\color{1}2,005}\\ 
    \hline
		{\tt Schrans-Troost-h} & {\color{1}0} & {\color{1}0} & {\color{1}0}\\ 
		\hline
	\end{tabular}
	\par
\end{centering}
\caption{Number of zero reductions computed by the algorithms.}
\label{tab:ch5:zero}
\end{table}

\section{Conclusion}
\label{sec:conclusion}
This paper contributes a more efficient usage, generalization, and combination
of, for specific incremental signature-based algorithms, known improvements. Even in situations where
it does not enlarge the number of detected useless critical pairs ($\ffc$,
$\ffe$) it gives quite impressive speed-ups using a faster,
less complex way of detecting useless data. 

Looking explicitly at $\ggv$, the improvement in terms of removing redundant 
critical pairs is astonishing. Due to the fact that $\ggv$ lacks a real 
implementation of (RW) the idea presented in Section~\ref{sec:idea} gives an 
easy way to add, at least partly, the strengths of $\ff$'s (RW) implementation 
to $\ggv$ without making the algorithm's description more complex. 

We have presented a rather neat improvement that can
be added to any existing implementation of an incremental signature-based 
algorithm without any bigger effort. Besides
having a huge impact on the computations for non-regular sequences, it does not 
slow down the overall efficiency of incremental signature-based Gr\"obner basis 
algorithms for regular input at all.
\section*{Acknowledgements}
I would like to thank the referees for improving this paper with their
suggestions, especially the one pointing out an error in a previous version of 
this paper, a fact Example~\ref{ex:problem} mirrors.
Furthermore, I would like to thank John Perry for helpful discussions and comments.
\bibliographystyle{plain}

%\bibliography{bib.bib}
\end{document}